\documentclass[a4paper,10pt]{amsart}

\usepackage[english]{babel}
\usepackage{amssymb}
\usepackage[sort]{natbib}
\usepackage{mathrsfs}
\usepackage{mathscinet}
\usepackage{fullpage}
\usepackage[colorlinks,linkcolor=blue,citecolor=blue]{hyperref}
\usepackage{kamil}

\def\K#1{\textcolor{black}{#1}}
\def\k2#1{\textcolor{black}{#1}}

\begin{document}
\bibliographystyle{plainnat}
\setcitestyle{numbers}

\title{An Erd\"os--R\'ev\'esz type law of the iterated logarithm for
reflected fractional Brownian motion}

\author{K.\ D\polhk{e}bicki}
\address{Mathematical Institute, University of
Wroc\l aw, pl.\ Grunwaldzki 2/4, 50-384 Wroc\l aw, Poland.}
\email{Krzysztof.Debicki@math.uni.wroc.pl}

\author{K.M.\ Kosi\'nski}
\address{Mathematical Institute, University of
Wroc\l aw, pl.\ Grunwaldzki 2/4, 50-384 Wroc\l aw, Poland.}
\email{Kamil.Kosinski@math.uni.wroc.pl}



\date{\today}

\begin{abstract}
Let $B_H=\{B_H(t):t\in\mathbb R\}$ be \K{a} fractional Brownian motion with Hurst parameter $H\in(0,1)$.
For the stationary storage process $Q_{B_H}(t)=\sup_{-\infty<s\le t}(B_H(t)-B_H(s)-(t-s))$, \K{$t\ge0$}, we
provide a tractable criterion for assessing whether, for any positive, non-decreasing function $f$,
$ \K{\mathbb P(Q_{B_H}(t) > f(t)\, \text{ i.o.})}$
equals 0 or 1. Using this criterion we find that, for a family of functions
$f_p(t)$, such that $z_p(t)=\mathbb P(\sup_{s\in[0,f_p(t)]}Q_{B_H}(s)>f_p(t))/f_p(t)=\mathscr C(t\log^{1-p} t)^{-1}$, for some $\mathscr C>0$,
$\K{\mathbb P(Q_{B_H}(t) > f_p(t)\, \text{ i.o.})= 1_{\{p\ge 0\}}}$.
Consequently, with $\xi_p (t) = \sup\{s:0\le s\le t, Q_{B_H}(s)\ge f_p(s)\}$, for $p\ge 0$,
$\lim_{t\to\infty}\xi_p(t)=\infty$ and $\limsup_{t\to\infty}(\xi_p(t)-t)=0$ a.s.
Complementary, we prove an Erd\"os--R\'ev\'esz type law of the iterated logarithm lower bound on $\xi_p(t)$, i.e.,
$\liminf_{t\to\infty}(\xi_p(t)-t)/h_p(t) = -1$ a.s., $p>1$; $\liminf_{t\to\infty}\log(\xi_p(t)/t)/(h_p(t)/t) = -1$ a.s., $p\in(0,1]$, where $h_p(t)=(1/z_p(t))p\log\log t$.
\end{abstract}

\maketitle

\section{Introduction and Main Results}
\label{sec:intro}

The analysis of properties of \textit{reflected} stochastic processes, being
developed in the context of classical \textit{Skorokhod problems}
and \K{their} applications to queueing theory, risk
theory and financial mathematics, is an actively investigated field of
applied probability. In this paper we analyze 0-1 properties
of a class of such processes, that due to its importance in queueing theory (and dual risk theory)
gained substantial interest;
see, e.g., \citep{Norros94,Piterbarg01,Asmussen03,Asmussen10}
or novel works on $\gamma$-reflected Gaussian processes \citep{HJP13,LHJ15}.

Consider a reflected (at 0) fractional Brownian motion with drift $Q_{B_H} = \{Q_{B_H}(t):t\ge0\}$,
given by the following formula
\begin{equation}
\label{prob.0}
Q_{B_H}(t)=B_H(t)-ct+\max\left( Q_{B_H}(0),-\inf_{s\in[0,t]}(B_H(s)-cs)   \right),
\end{equation}
where
$c>0$ and
$B_H=\{B_H(t):t\in\rr\}$ is a fractional Brownian motion (fBm) with Hurst
parameter $H\in(0,1)$, i.e., a centered Gaussian process
with covariance function
$
\Cov(B_H(t), B_H(s)) = \half\left(|t|^{2H}+|s|^{2H}-|t-s|^{2H}\right).
$
We focus on the investigation of the long-time behavior of
the unique stationary solution of \eqref{prob.0}, which has the following representation
\begin{equation}
\label{eq:stat1}
Q_{B_H}(t)=\sup_{-\infty< s\le t}\left(B_H(t)-B_H(s)-c(t-s)\right).
\end{equation}
With no loss of generality in the reminder of \K{this} paper
we assume that the drift parameter $c\equiv 1$.
An important stimulus to analyze the distributional properties of
$Q_{B_H}$ and its functionals stems from \K{the Gaussian fluid} queueing theory, where
the stationary buffer content process
in a queue which is fed by $B_H$ and emptied with constant rate $c=1$
is described by \eqref{eq:stat1}; see e.g. \citep{Norros94}.
In particular, in the seminal paper by \citet{Husler99}
the exact asymptotics of one dimensional marginal distributions of $Q_{B_H}$ was derived;
see also \citep{Dieker05a,Debicki02,DeJ16} for results on more general
Gaussian input processes.

The purpose of this paper is
to investigate the asymptotic 0-1 behavior of the processes
$Q_{B_H}$. Our first contribution is an \K{analog} of the classical finding of
\citet{Watanabe70}, where an asymptotic 0-1 type of behavior for centered stationary
Gaussian processes was analyzed.
\begin{theorem}
\label{thm:equiv}
For all functions $f(t)$ that are positive and nondecreasing on some interval $[T,\infty)$, it follows that
\[
\prob{Q_{B_H}(t) > f(t)\quad\text{\rm{i.o.}}} = 0\quad \text{or}\quad 1,
\]
according as the integral
\[
\mathscr I_f :=
\int_T^\infty \frac{1}{f(u)}\prob{\sup_{t\in[0,f(u)]} Q_{B_H}(t) > f(u)}\D u
\]
{is finite or infinite}.
\end{theorem}
\noindent The exact asymptotics, as $u$ grows large, of the probability in $\mathscr I_f$ \K{was}
found by \citet[Theorem 7]{Piterbarg01}.
\K{Namely,} for any $T>0$,
\begin{equation}
\label{eq:actualAsympt}
\prob{\sup_{t\in[0,T f(u)]} Q_{B_H}(t) > f(u)}
=
\sqrt\pi a^\frac{2}{H} b^{-\half}\mathcal H_{B_H}^2  T(v_f(u))^{\frac{2}{H}-1}\Psi(v_f(u))(1+o(1)), \as u,
\end{equation}
where $v_f(u) = A f^{1-H}(u)$, $\Psi(u)=1-\Phi(u)$,
$\Phi$ is the distribution function of \K{the} unit normal law and the constants $a, b, A, \mathcal H_{B_H}$ are given explicitly in
\autoref{sec:prelim}.
Since relation \eqref{eq:actualAsympt} also holds when $T=T(u)\to 0$, provided that $T(u) (f(u))^{(1-H)/H}\toi$,
\K{we have} that for $H\in(0,\half)$, as $u\toi$,
\[
\frac{1}{f(u)}\prob{\sup_{t\in[0,f(u)]} Q_{B_H}(t) > f(u)}\sim
\prob{\sup_{t\in[0,1]} Q_{B_H}(t) > f(u)}.
\]

\autoref{thm:equiv} provides a tractable criterion for settling the dichotomy of $\prob{Q_{B_H}(t) > f(t)\ \text{\rm{i.o.}}}$.
For instance, \K{let} $C_H = (2(1-H)^2-H)/(2H(1-H))$ and
\begin{equation}
\label{eq:defp}
f_p(s) = \left(\frac{2}{A^2}\left(\log s + \left(1+C_H-p\right)\log_2 s\right)\right)^{\frac{1}{2(1-H)}},\quad p\in\rr, H\in(0,1).
\end{equation}
One \K{can check} that, as $u\toi$,
\begin{equation}
\label{eq:Gf}
\frac{1}{f_p(u)}\prob{ \sup_{t\in[0,f_p(u)]} Q_{B_H}(t) > f_p(u)}
=
\frac{a^\frac{2}{H} b^{-\half}}{\sqrt 2}\mathcal H_{B_H}^2 A^{\frac{1}{1-H}}2^{C_H} (u \log^{1-p} u)^{-1}(1+o(1)).
\end{equation}
Hence, for any $p\in\rr$,
\[
\prob{Q_{B_H}(t) > f_p (t)\quad\text{i.o.}}=
\left\{
\begin{array}{cc}
1 & \text{if } p\ge 0, \\
0 & \text{if } p<0.  \\
\end{array}
\right.
\]
\begin{corollary}
\label{cor:main}
For any $H\in(0,1)$,
\[
\limsup_{t\toi} \frac{Q_{B_H}(t)}
                     {\left(
 \log t
\right)^{\frac{1}{2(1-H)}}} =\left(\frac{2}{A^2}\right)^{\frac{1}{2(1-H)}}
\quad\text{a.s.}
\]
\end{corollary}
\noindent This result extends findings of \citet[Theorem 1]{GlZ00}, where it was proven that
the above convergence holds weakly as well as in $L_p$ for all $p\in[1,\infty)$.

Now consider the process $\xi_p=\{\xi_p(t):t\ge 0\}$ defined as
\[
\xi_p(t)=\sup\{s:0\le s\le t, Q_{B_H}(s)\ge f_p(s)\}.
\]
Since $\mathscr I_{f_p} = \infty$ for $p\ge 0$, from \autoref{thm:equiv} it follows that
\[
\lim_{t\toi} \xi_p(t) = \infty\quad\text{a.s.} \quad\text{and}\quad
\limsup_{t\toi}(\xi_p(t) - t) = 0\quad\text{a.s.}
\]
Let, cf. \eqref{eq:Gf},
\[
h_p(t)= p\left(\frac{1}{f_p(t)}\prob{ \sup_{s\in[0,f_p(t)]} Q_{B_H}(s) > f_p(t) }\right)^{-1}\log_2 t.
\]

The second contribution of this paper is an Erd\"os--R\'ev\'esz type of law of the iterated logarithm for the process $\xi_p$. We refer to \citet{Sha92} for more background and references on
Erd\"os--R\'ev\'esz type law of the iterated logarithm
and a related result for centered stationary
Gaussian processes; \K{see also} \citet{Debicki17} for extensions to order statistics.
\begin{theorem}
\label{thm:main}
If $p>1$, then
\[
\liminf_{t\toi}\frac{\xi_p(t)-t}{h_p(t)} = - 1\ \ {\rm a.s.}
\]
If $p\in(0,1]$, then
\[
\liminf_{t\toi}\frac{\log\left(\xi_p(t)/t\right)}{h_p(t)/t } = - 1\ \ {\rm a.s.}
\]
\end{theorem}

Now, let us complementary put $\eta_p = \{\eta_p(t):t\ge 0\}$, where
\[
\eta_p(t) = \inf\{s\ge t: Q_{B_H}(s)\ge f_p(s)\}.
\]
Since
\[
\prob{\xi_p(t) - t\le - x} = \prob{\sup_{s\in(t-x,t]}\frac{Q_{B_H}(s)}{f_p(s)}< 1}\]
and
\[
\prob{z - \eta_p(z)\le - x} = \prob{\sup_{s\in[z,z+x]}\frac{Q_{B_H}(s)}{f_p(s)}< 1},
\]
then it follows that
\begin{equation}
\label{eq:eq}
\liminf_{t\toi}\frac{\xi_p(t)-t}{h_p(t)}=\liminf_{z\toi}\frac{z-\eta_p(z)}{h_p(z)}.
\end{equation}

\autoref{thm:main} shows that for $t$ big enough,
there exists an $s$ in $[t - h_p(t), t]$ (as well as in $[t, t+ h_p(t)]$ by \eqref{eq:eq}) such that
$Q_{B_H}(s)\ge f_p(s)$ and that the length \K{$h_p(t)$ of the interval
is the smallest} possible.
This shines new light on \K{results, which are} intrinsically connected with Gumbel limit theorems; see, e.g.,
\citep{Leadbetter83}, where the function $h_p(t)$ plays crucial role. We shall pursue this elsewhere.

The paper is organized as follows.
In \autoref{sec:prelim} we introduce some useful properties of storage processes fed by fractional Brownian motion. In \autoref{sec:auxLem} we provide a collection of basic results on how to interpret extremes of the storage process $Q_{B_H}$ as extremes of a Gaussian field related to the
fractional Brownian motion $B_H$. Furthermore, \K{in \autoref{sec:Proofs} we prove lemmas,}
which constitute building blocks of the proofs of the main results.

\section{Properties of the storage process}
\label{sec:prelim}
\K{
In this section we introduce} some notation and state some
properties of the supremum of the  process $Q_{B_H}$ as derived in \citep{Piterbarg01, Husler04b}. We begin with the relation
\begin{equation}
\label{eq:basic_transformation}
\prob{\sup_{t\in[0,T]} Q_{B_H}(t)>u} =\prob{ \sup_{\substack{s\in[0,T/u] \\ \tau\ge0}} Z_u(s,\tau)>u^{1-H}},\quad\text{for any}\quad T>0,
\end{equation}
where, with $\nu(\tau)=\tau^{-H}+\tau^{1-H}$,
\[
Z_{u}(s,\tau):= \frac{B_H(u(\tau+s))-B_H(su)}{\tau^Hu^H\nu(\tau)}
\]
is a Gaussian field. Note that the self-similarity property of $B_H$ implies that
the field $Z_u$ has the same distribution for any $u$.
Thus, we do not use $u$ as an additional parameter in the
\K{following} notation whenever it is not needed; \K{let} $Z(s,\tau):= Z_1(s,\tau)$. Furthermore, the field $Z(s,\tau)$ is stationary in $s$, but not in $\tau$.
The variance $\sigma_Z^2(\tau)$ of the field $Z(s,\tau)$ equals $\nu^{-2}(\tau)$ and
$\sigma_Z(\tau)$ has a single maximum point at \[\tau_0 = \frac{H}{1-H}.\]
Taylor \K{expansion leads to}
\[
\sigma_Z(\tau)=\frac{1}{A}-\frac{B}{2A^2}(\tau-\tau_0)^2 + O((\tau-\tau_0)^3),
\]
as $\tau\to\tau_0$, where
\begin{align*}
A&= \frac{1}{1-H}\left(\frac{H}{1-H}\right)^{-H}=\nu(\tau_0),\\
B&= H\left(\frac{H}{1-H}\right)^{-H-2}=\nu''(\tau_0).\\
\end{align*}

Let us define the correlation function of the process $Z_u$ as follows
\begin{align}
r_{u,u'}&(s,\tau,s',\tau')
:= \mathbb E Z_u(s,\tau) Z_u'(s',\tau')\nu(\tau)\nu(\tau')\nonumber\\
\label{eq:rdef}
&= \frac{|us-u's'|^{2H}}{2(u\tau u'\tau')^{H}}
\left(
\left|1+\frac{u\tau}{(us-u's')}\right|^{2H}-\left|1+\frac{(u\tau-u'\tau')}{(us-u's')}\right|^{2H}
+ \left|1-\frac{u'\tau'}{(us-u's')}\right|^{2H} - 1\right).
\end{align}
By series expansion we find for any fixed $\tau_1<\tau_0<\tau_2$ and $\tau,\tau'$ with
$0<\tau_1<\tau,\tau'<\tau_2<\infty$,
\[
|r_{u,u'}(s,\tau,s',\tau')| \le
\frac{|us-u's'|^{2H}}{(u\tau u'\tau')^{H}}2H|2H-1||us-u's'|^{-2}(u\tau u'\tau'),\quad 2H\ne 1,
\]
provided that $|\frac{u}{us-u's'}|$ and $|\frac{u'}{us-u's'}|$ are sufficiently small. For $2H=1$, we have $r_{u,u'}(s,\tau,s',\tau')=0$ since the increments of Brownian motion on disjoint intervals are independent. Therefore,

\begin{equation}
\label{eq:defr*}
r^*(t) := \sup_{\substack{\left|\frac{us-u's'}{u}\right|,\left|\frac{us-u's'}{u'}\right|\ge t\\ \tau_1<\tau, \tau'<\tau_2\\u,u',s,s'>0}} |r_{u,u'}(s,\tau,s',\tau')|
\le
\mathcal K t^{-\lambda},
\end{equation}
for $\lambda = 2-2H >0$, $t$ sufficiently large and some positive constant $\mathcal K$ depending only on $H$, $\tau_1$ and $\tau_2$.
Similarly, from \eqref{eq:rdef} it follows that for any fixed $M$ there exists $\delta\in(0,1)$ such that
\begin{equation}
\label{eq:asymptotic}
0<\delta\le\inf_{\substack{\left|\frac{us-u's'}{u}\right|,\left|\frac{us-u's'}{u'}\right|\le M\\|\tau-\tau^*|,|\tau'-\tau^*|\le m}} r_{u,u'}(s,\tau,s',\tau')\le
\sup_{\substack{\left|\frac{us-u's'}{u}\right|,\left|\frac{us-u's'}{u'}\right|\le M\\|\tau-\tau^*|,|\tau'-\tau^*|\le m}} r_{u,u'}(s,\tau,s',\tau')\le 1-\delta<1,
\end{equation}
for sufficiently small $m$.

\subsection{Asymptotics}
Due to the following lemma, while analyzing tail asymptotics of the supremum of $Z$, we can restrict the considered domain of $(s,\tau)$ to a strip with $|\tau-\tau_0|\le \log v/v$.
\begin{lem}[\citet{Piterbarg01}, Lemma 2 and 4]
\label{lem:asymptotics}
There exists a positive constant $C$ such that for any $v,T > 0$,
\begin{equation}
\label{eq:tauBigger}
\prob{ \sup_{\substack{s\in[0,T] \\ |\tau-\tau_0|\ge \log v/v}} A Z(s,\tau)>v}
\le
C T v^{2/H}\exp\left(-\half v^2 - b \log^2 v\right),
\end{equation}
where $b=B/(2A)$. Furthermore,
for any $T>0$, with $a=1/(2\tau_0^{2H})$, as $v\toi$,
\[
\prob{ \sup_{\substack{s\in[0,T] \\ |\tau-\tau_0|\le \log v/v}} A Z(s,\tau)>v}
=
\sqrt\pi a^\frac{2}{H} b^{-\half}\mathcal H_{B_H}^2 T v^{\frac{2}{H}-1}\Psi(v)(1+o(1)),
\]
where
\[
\mathcal H_{B_H}  = \lim_{T\toi} T^{-1}\ee \exp\left(\sup_{t\in[0,T]}\left(\sqrt 2 B_H(t) - t^{2H}\right)\right)
\in(0,\infty),
\]
is the so-called \textit{Pickands' constant}.
This holds also for $T= v^{-1/H'}$, with $1>H'>H$.
\end{lem}
\citet[Corollary 2]{Husler04} showed that the above actually holds true for $T$ depending on $v$ such that
$v^{-1/H'}<T<\exp(c v^2)$, for any $H'\in(H,1)$ and $c\in(0,\half)$.

\subsection{Discretization}
\label{sec:disc}
Let $\tau^*(v)=\log v/ v$ and $J(v)=\{\tau: |\tau-\tau_0|\le \tau^*(v)\}$. For a fixed $T,\theta>0$ and some $v>0$, let us define a discretization of the set $[0,T]\times J(v)$ as follows
\begin{align*}
s_l &= lq(v),\quad 0\le l\le L,\quad L = [T/q(v)],\quad q(v) = \theta v^{-\frac{1}{H}},\\
\tau_n &= \tau_0+ nq(v),\quad 0\le |n|\le N,\quad N = [\tau^*(v)/q(v)].
\end{align*}
Along the same lines as in \citep[Lemma 6]{Husler04b} we get the following lemma.
\begin{lem}
\label{lem:discreate_approx}
There exist positive constants $K_1,K_2,v_0>0$, such that, for any $\theta>0$ and $v\ge v_0$,
\begin{align*}
\mathbb P &\left(\max_{\substack{0\le l\le L\\0\le |n|\le N}} AZ(s_l,\tau_n)\le
v - \frac{\theta^{\frac{H}{2}}}{v},
\sup_{\substack{s\in[0,T]\\\tau\in J(v)}} AZ(s,\tau)>v\right)
\le
K_1v^{\frac{2}{H}-1} \Psi(v) \theta^{\frac{H}{2}}
\exp\left(-\theta^{-H}/K_2\right).
\end{align*}
\end{lem}
Finally, it is possible to approximate tail asymptotics of the supremum of $Z$
on the strip $[0, T]\times J(v)$ by maximum taken over discrete time points.
The proof of the following lemma follows line-by-line \K{the same a}s the proof of \citep[Lemma 4]{Piterbarg01} and thus we omit it. Similar result can be found in, e.g., \citep[Lemma 7]{Husler04b}.
\begin{lem}
\label{lem:disc_asymp}
For any $T, \theta>0$, as $v\toi$,
\[
\prob{ \max_{\substack{0\le l\le L\\0\le |n|\le N}} AZ(s_l,\tau_n) > v}
=
\sqrt\pi a^\frac{2}{H} b^{-\half}
\left(
\mathcal H_{B_H}^\theta
\right)^2 T v^{\frac{2}{H}-1}\Psi(v)(1+o(1)),
\]
where
$\mathcal H_{B_H}^\theta=\lim_{S\to\infty}
S^{-1}\ee \exp\left(\sup_{t\in\theta \mathbb Z \cap [0,S]}\left(\sqrt 2 B_H(t) - t^{2H}\right)\right)$.
\end{lem}
It follows easily that $\mathcal H_{B_H}^\theta\to \mathcal H_{B_H}$ as $\theta \to 0$,
so that the above asymptotics \K{is} the same as in \autoref{lem:asymptotics}
when the discretization parameter $\theta$ decreases to zero so that the number of discretization points grows to infinity.

\section{Auxiliary Lemmas}
\label{sec:auxLem}
We begin with some auxiliary lemmas that are later needed in the proofs.
The first lemma is \K{a slightly modified version of} \citep[Theorem 4.2.1]{Leadbetter83}.
\begin{lem}[Berman's inequality]
Suppose \K{that}
$\xi_1,\ldots,\xi_n$ are 
normal \K{random} variables with
\K{correlation} matrix $\Lambda^1=(\Lambda^1_{i,j})$
and
$\eta_1,\ldots,\eta_n$ similarly with \K{correlation}
matrix $\Lambda^0=(\Lambda_{i,j}^0)$.
Let
\K{$\sigma(\xi_i)=\sigma(\eta_i)\in (0,1]$,}
$\rho_{i,j}=\max(|\Lambda_{i,j}^1|,|\Lambda_{i,j}^0|)$
\K{and $u_i$ be real numbers, $i=1,\ldots,n$}.
Then,
\begin{align*}
\mathbb P\left(
\bigcap_{j=1}^n
\{\xi_j\le u_j\}\right) &- \prob{\bigcap_{j=1}^n\{\eta_j\le u_j\}}
\\
&\le
\frac{1}{2\pi}\sum_{1\le i<j\le n}\left(\Lambda_{i,j}^1-\Lambda_{i,j}^0\right)^+
(1-\rho_{i,j}^2)^{-\half}\exp\left(-\frac{u_i^2+u_j^2}{2(1+\rho_{i,j})}\right).
\end{align*}
\end{lem}
The following lemma is a general form of the Borel-Cantelli lemma; cf. \citep{Spitzer64}.
\begin{lem}[Borel-Cantelli lemma]
Consider a sequence of \K{events} $\{E_k\}_{k=0}^\infty$. If
\[
\sum_{k=0}^\infty \prob{E_k} < \infty,
\]
then $\prob{E_n\text{ i.o.}} = 0$. Whereas, if
\[
\sum_{k=0}^\infty \prob{E_k} = \infty\quad\text{and}\quad
\liminf_{n\toi}\frac{\sum_{1\le k\ne t\le n}\prob{E_k E_t}}{\left(\sum_{k=1}^n\prob{E_k}\right)^2}\le 1,
\]
then $\prob{E_n\text{ i.o.}} = 1$.
\end{lem}

\begin{lem}
\label{lem:bound}
For any $\varepsilon\in(0,1)$, there exist positive constants $K$ and $\rho$ depending only on $\varepsilon, H, p$ and $\lambda$ such that
\[
\prob{\sup_{S< t\le T} \frac{Q_{B_H}(t)}{f_p(t)}\le 1}
\le
\exp
\left(
-\frac{(1-\varepsilon)}{(1+\varepsilon)}
\int_{S+f_p(S)}^{T}
\frac{1}{f_p(u)}
\prob{\sup_{t\in [0,f_p(u)]} Q_{B_H}(t) >  f_p(u)}\D u
\right) + K S^{-\rho},
\]
for any $T-f_p(S)\ge S\ge K$, with $f_p(T)/f_p(S)\le \mathcal C$ and $\mathcal C$ being some universal positive constant.
\end{lem}
\begin{proof}
Let $\varepsilon\in(0,1)$ be some positive constant. For the reminder of the proof let $K$ and $\rho$ be two positive constants depending only on $\varepsilon, H, p$ and $\lambda$ that may differ from line to line. For any $k\ge 0$ put $s_0 = S$, $y_0 = f_p(s_0)$, $t_0 = s_0 + y_0$, $x_0 = f_p(t_0)$ and
\begin{align}
\nonumber
s_k & = t_{k-1} + \varepsilon x_{k-1},
\quad y_k = f_p(s_k),
\quad t_k = s_k + y_k,
\quad x_k = f_p(t_k),\\
\label{eq:defv}
\quad I_k &=(s_k,t_k],
\quad v_k = A x_k^{1-H} = v_{f_p}(t_k),
\quad \tilde I_k = \frac{I_k}{x_k} = (\tilde s_k, \tilde t_k],
\quad |\tilde I_k| = \frac{y_k}{x_k}.
\end{align}
From this construction, it is easy to see that the intervals $I_k$ are disjoint. Furthermore, $\delta(I_k,I_{k+1})=\varepsilon x_k$, and $1-\varepsilon\le y_k/x_k\le 1$, for any $k\ge0$ and sufficiently large $S$.
Note that, for any $k\ge 0$, $|I_k|\sim f_p(S)$ as $S$ grows large,
therefore if $T(S,\varepsilon)$ is the smallest number of intervals $\{I_k\}$ needed to cover $[S,T]$, then $T(S,\varepsilon)\le[(T-S)/(f_p(S)(1+\varepsilon))]$. Moreover, since $f_p(T)/f_p(S)$ is bounded by the constant $\mathcal C>0$ not depending on $S$ and $\varepsilon$, it follows that, $x_k/x_t\le \mathcal C$ for any $0\le t<k\le T(S,\varepsilon)$.

Now let us introduce a discretization of the set $\tilde I_k\times J(v_k)$ as in \autoref{sec:disc}.
That is, for some $\theta>0$, define grid points
\begin{align*}
\label{eq:disc100}
s_{k,l} &= \tilde s_k + lq_k,\quad 0\le l\le L_k,\quad  L_k = [(1-\varepsilon)/ q_k],
\quad q_k  = \theta v_k^{-\frac{1}{H}},\\
\tau_{k,n} &= \tau_0+ nq_k, \quad 0\le |n| \le N_k,\quad N_k = [\tau^*(v_k)/q_k].
\end{align*}
Since $f_p$ is an increasing function, it easily follows that,
\begin{align*}
\mathbb P  &\Bigg(\sup_{S< t\le T}\frac{Q_{B_H}(t)}{f_p(t)}\le 1\Bigg)
\le
\prob{\bigcap_{k=0}^{T(S,\varepsilon)}\left\{\sup_{t \in I_k} Q_{B_H}(t) \le x_k\right\}}
\le
\prob{\bigcap_{k=0}^{T(S,\varepsilon)}\left\{\sup_{\substack{s \in I_k/x_k \\ \tau\in J(v_k)}} AZ_{x_k}(s,\tau) \le v_k\right\}}\\
&\le
\prob{\bigcap_{k=0}^{T(S,\varepsilon)}\left\{\max_{\substack{ 0\le l\le L_k\\\quad 0\le |n| \le N_k}} AZ_{x_k}(s_{k,l},\tau_{k,n})\le v_k\right\}}\\
&\le
\prod_{k=0}^{T(S,\varepsilon)}\prob{\max_{\substack{ 0\le l\le L_k\\\quad 0\le |n| \le N_k}} AZ_{x_k}(s_{k,l},\tau_{k,n})\le v_k}
+\sum_{0\le t<k\le T(S,\varepsilon)}C_{k,t} =: P_1 + P_2,
\end{align*}
where the last inequality follows from Berman's inequality with
\[
C_{k,t}=
 \sum_{\substack{0\le l\le L_k\\0\le p\le L_t}}\sum_{\substack{|n|\le N_k\\|m|\le N_t}}
\frac{|r_{x_k,x_t}(s_{k,l},\tau_{k,n},s_{t,p},\tau_{t,m})|}{\sqrt{1- r^2_{x_k,x_t}(s_{k,l},\tau_{k,n},s_{t,p},\tau_{t,m})}}
\exp\left(-\frac{\half(v_k^2+v_t^2)}{1+|r_{x_k,x_t}(s_{k,l},\tau_{k,n},s_{t,p},\tau_{t,m})|}\right).
\]

\noindent\textit{Estimate of $P_1$.}

\vb

Note that we can use the fact that $Z_{x_k}$ has the same distribution as
$Z_1\equiv Z$ for any $x_k$.
Since the process $Z$ is stationary \K{with respect to} the first variable, from \autoref{lem:disc_asymp}, for any $\varepsilon\in(0,1)$, sufficiently large $S$ and small $\theta$,
\begin{align*}
P_1 &\le
\exp\left(
-\sum_{k=0}^{T(S,\varepsilon)} \prob{\max_{\substack{ 0\le l\le L_k\\\quad 0\le |n| \le N_k}} AZ_{x_k}(s_{k,l},\tau_{k,n}) >  v_k}
\right)\\
&\le
\exp\left(
-(1-\frac{\varepsilon}{4})\sum_{k=0}^{T(S,\varepsilon)} \prob{\sup_{(s,\tau)\in \tilde I_k\times J(v_k)} AZ (s,\tau) >  v_k}
\right)
\end{align*}
Then, by \eqref{eq:basic_transformation} combined with \eqref{eq:actualAsympt},
\begin{align*}
P_1
&\le
\exp\left(
-(1-\frac{\varepsilon}{2})\sum_{k=0}^{T(S,\varepsilon)} \prob{\sup_{\substack{s\in \tilde I_k\\\tau\ge 0}} AZ (s,\tau) >  v_k}
\right)\\
&=
\exp\left(
-(1-\frac{\varepsilon}{2})\sum_{k=0}^{T(S,\varepsilon)} \prob{\sup_{t\in [0,\frac{y_k}{x_k}f_p(t_k)]} Q_{B_H}(t) >  f_p(t_k)}
\right)\\
&\le
\exp\left(
-(1-\varepsilon)\sum_{k=0}^{T(S,\varepsilon)} \prob{\sup_{t\in [0,f_p(t_k)]} Q_{B_H}(t) >  f_p(t_k)}
\frac{f_p(s_k)}{f_p(t_k)}
\right)\\
&\le
\exp\left(
-\frac{1-\varepsilon}{1+\varepsilon}\int_{S+f_p(S)}^{T}
\frac{1}{f_p(u)}\prob{\sup_{t\in [0,f_p(u)]} Q_{B_H}(t) >  f_p(u)}\D u
\right).
\end{align*}

\noindent\textit{Estimate of $P_2$.}

\vb

For any $0\le t<k\le T(S,\varepsilon)$, $0\le l\le L_k$, $0\le p\le L_t$, we have
\begin{align*}
x_ks_{k,l}- x_ts_{t,p}
&=
(s_k+ x_k l q_k) - (s_t + x_t p  q_t)\\
&=
\sum_{i=t}^{k-1}(y_i+\varepsilon x_i) + x_k l q_k - x_t p  q_t
\ge
\sum_{i=t}^{k-1}(y_i+\varepsilon x_i) - x_t(1-\varepsilon)\\
&\ge (y_t+\varepsilon x_t)(k-t)-x_t(1-\varepsilon)\ge x_t(k-t)\varepsilon,
\end{align*}
where the last inequality holds provided that $k-t\ge s_0$ with $s_0$ sufficiently large.
Therefore, c.f. \eqref{eq:defr*},
\[
r_{k,t}^*:=\sup_{
	\substack{		
		0\le l\le L_k,
		0\le p\le L_t\\
		|n|\le N_k,
		|m|\le N_t
	}
}
|r_{x_k,x_t}(s_{k,l},\tau_{k,n},s_{t,p},\tau_{t,m})|
\le
r^*((k-t)\varepsilon)\le
\mathcal K (k-t)^{-\lambda}
\le
\min(1,\lambda)/4.
\]
Moreover, from \eqref{eq:asymptotic} it follows that, for any $0\le k-t\le s_0$,
there exists a constant $\zeta\in(0,1)$ depending only on $\varepsilon$ such that for sufficiently large $S$,
\[
\sup_{
	\substack{		
		0\le l\le L_k,
		0\le p\le L_t\\
		|n|\le N_k,
		|m|\le N_t
	}
}
|r_{x_k,x_t}(s_{k,l},\tau_{k,n},s_{t,p},\tau_{t,m})|\le\zeta<1.
\]
Finally, recall that $N_k\le L_k \le \theta^{-1} v_k^{\frac{1}{H}}$ and $\exp(-v_k^2/2)= (t_k\log^{(1+C_H-p)} t_k )^{-1}$, c.f. \eqref{eq:defp},\eqref{eq:defv}, so that
\begin{align*}
P_2 &\le \frac{4}{\sqrt{1-\zeta^2}}
\sum_{0\le t<k\le T(S,\varepsilon)}
L_kL_tN_kN_t r_{k,t}^*
\exp\left(-\frac{v_k^2+v_t^2}{2(1+r_{k,t}^*)}\right)\\
&\le
K\left(
\sum_{\substack{0<k-t\le s_0 \\ 0\le t<k\le T(S,\varepsilon)}} +
\sum_{\substack{k-t > s_0 \\ 0\le t<k\le T(S,\varepsilon)}}
\right) (\cdot)\\
&\le
K\Bigg(
\sum_{k=0}^\infty v_k^{\frac{4}{H}}
\exp\left(-\frac{ v_k^2}{1+\zeta}\right)
+
\sum_{\substack{k-t > s_0 \\ 0\le t<k\le T(S,\varepsilon)}}
v_k^{\frac{2}{H}}
v_t^{\frac{2}{H}}
(k-t)^{-\lambda}
\exp\left(-\frac{v_k^2+v_t^2}{2(1+\frac{\lambda}{4})}\right)
\Bigg)\\
&\le
K\left(
\sum_{k=0}^\infty
t_k^{-\frac{2}{1+ \sqrt{\zeta}}}
+
\sum_{\substack{k-t > s_0 \\ 0\le t<k\le T(S,\varepsilon)}}
t_k^{-\frac{1}{1+\frac{\lambda}{2}}}
 t_t^{-\frac{1}{1+\frac{\lambda}{2}}}
(k-t)^{-\lambda}
\right)\\
&\le
K\left(
\sum_{k=[S]}^\infty
k^{-\frac{2}{1+ \sqrt{\zeta}}}
+
\sum_{\substack{[S]\le t<k\le \infty}}
k^{-\frac{1}{1+\frac{\lambda}{2}}}
t^{-\frac{1}{1+\frac{\lambda}{2}}}
(k-t)^{-\lambda}
\right)\\
&\le
K S^{-\rho},
\end{align*}
where the last inequality follows from basic algebra.
\end{proof}
Let $S>0$ be any fixed number, $a_0 = S$, $y_0 = f_p(a_0)$ and $b_0=a_0+y_0$. For $i>0$, define
\begin{equation}
\label{eq:def_ai2}
a_i = b_{i-1},\quad y_i = f_p(a_i), \quad b_i = a_i + y_i,\quad M_i = (a_i, b_i],
\quad v_i = Ay_i^{1-H},\quad \tilde M_i = \frac{M_i}{y_i} = (\tilde a_i, \tilde b_i].
\end{equation}
From this construction it is easy to see that the intervals $M_i$ are disjoint, $\cup_{j=0}^{i} M_j=(S,b_{i}]$ and $|\tilde M_i| = 1$.
Now let us introduce a discretization of the set $\tilde M_i\times J(v_i)$ as in \autoref{sec:disc}. That is, for some $\theta>0$, define grid points
\begin{align}
\label{eq:disc2}
s_{i,l} &= \tilde a_i+ l q_i,\quad 0\le l\le L_i,\quad  L_i = [1/ q_i],
\quad q_i  = \theta v_i^{-\frac{1}{H}},\\
\nonumber
\tau_{i,n} &= \tau_0+ nq_i, \quad 0\le |n| \le N_i,\quad N_i = [\tau^*(v_i)/q_i].
\end{align}
With the above notation, we have the following lemma.
\begin{lem}
\label{lem:lbound}
For any $\varepsilon\in(0,1)$ there \K{exist} positive constants $K$ and $\rho$ depending only on $\varepsilon, H, p$ and $\lambda$ such that, with $\theta_i= v_i^{-4/H}$,
\begin{align*}
\mathbb P &\left(
\bigcap_{i=0}^{[(T-S)/f_p(S)]}\left\{
\max_{\substack{0\le l\le L_i\\0\le |n| \le N_i}} A Z_{y_i}(s_{i,l},\tau_{i,n}) \le v_i - \frac{\theta_i^{\frac{H}{2}}}{v_i}
	\right\}
	\right)
	\\
&\ge
\frac{1}{4}\exp\left(
-(1+\varepsilon)
\int_{S}^{T}
\frac{1}{f_p(u)}\prob{\sup_{t\in [0,f_p(u)]} Q_{B_H}(t) >  f_p(u)}
	\D u
\right) - K S^{-\rho},
\end{align*}
for any $T-f_p(S)\ge S\ge K$, with $f_p(T)/f_p(S)\le \mathcal C$ and $\mathcal C$ being some universal positive constant.
\end{lem}
\begin{proof}
Put $\hat v_i =  v_i - \theta_i^{\frac{H}{2}} / v_i$ and $I = [(T-S)/f_p(S)]$. Similarly as in the proof of \autoref{lem:bound} we find that Berman's inequality implies
\begin{align*}
\mathbb P &\left(
\bigcap_{i=0}^{I}\left\{
\max_{\substack{0\le l\le L_i\\0\le |n| \le N_i}} A Z_{y_i}(s_{i,l},\tau_{i,n}) \le v_i - \frac{\theta_i^{\frac{H}{2}}}{v_i}
	\right\}
	\right)\\
&\ge
\prod_{i=0}^{I} \prob{\max_{\substack{0\le l\le L_i\\0\le |n| \le N_i}} A Z_{y_i}(s_{i,l},\tau_{i,n}) \le
v_i - \frac{\theta_i^{\frac{H}{2}}}{v_i}}
 -
\sum_{0\le i<j\le I} D_{i,j}=:P_1'+P_2',
\end{align*}
where
\[
D_{i,j} = \frac{1}{2\pi}\sum_{\substack{0\le l\le L_i\\0\le p\le L_j}}\sum_{\substack{|n|\le N_i\\|m|\le N_j}}
\frac{(\tilde r_{y_i,y_j}(s_{i,l},\tau_{i,n},s_{j,p},\tau_{j,m}))^+}{\sqrt{1-\tilde r^2_{y_i,y_j}(s_{i,l},\tau_{i,n},s_{j,p},\tau_{j,m})}}\
\exp\left(-\frac{\half(\hat v_i^2+\hat v_j^2)}{1+|\tilde r_{y_i,y_j}(s_{i,l},\tau_{i,n},s_{j,p},\tau_{j,m})|}\right),
\]
with
\[
\tilde r_{y_i,y_j}(s_{i,l},\tau_{i,n},s_{j,p},\tau_{j,m})
= - r_{y_i,y_j}(s_{i,l},\tau_{i,n},s_{j,p},\tau_{j,m}).
\]

\noindent\textit{Estimate of $P_1'$.}

\vb

By \autoref{lem:asymptotics} the correction term $\theta_i^{\frac{H}{2}}/v_i$ does
not change the order of \K{the} asymptotics of the tail of $Z$. Furthermore, the tail asymptotics of the supremum on the strip
$(s,\tau)\in\tilde M_i\times J(v_i)$ are of the same order if $\tau\ge 0$.
Hence, for every $\varepsilon>0$, \k2{following
the same lines of reasoning as in the estimation of $P_1$ in \autoref{lem:bound},}
\begin{align*}
P_1' &\ge
\prod_{i=0}^{I} \left(1-\prob{\max_{\substack{0\le l\le L_i\\0\le |n| \le N_i}} A Z_{y_i}(s_{i,l},\tau_{i,n}) > \hat v_i} \right)
\ge
\frac{1}{4}
 \exp\left(
-\sum_{i=0}^{I} \prob{\max_{\substack{0\le l\le L_i\\0\le |n| \le N_i}} A Z_{y_i}(s_{i,l},\tau_{i,n}) > \hat v_i}
\right)\\
&\ge
\frac{1}{4}
 \exp\left(
-\sum_{i=0}^{I} \prob{\sup_{\substack{s\in\tilde M_i\\ \tau\in J(v_i)}} A Z(s,\tau) > v_i - \frac{\theta_i^{\frac{H}{2}}}{v_i}}
\right)
\\
&\ge
\frac{1}{4}
 \exp\left(
-(1+\varepsilon)\sum_{i=0}^{I} \prob{\sup_{\substack{s\in\tilde M_i\\ \tau\ge 0}} A Z(s,\tau) > v_i}
\right)
\\
&=
\frac{1}{4}
 \exp\left(
-(1+\varepsilon)\sum_{i=0}^{I} \prob{\sup_{t\in[0,f_p(a_i)]} Q_{B_H}(t) > f_p(a_i)}
\right)\\
&\ge
\frac{1}{4}\exp\left(
-(1+\varepsilon)
\int_{S}^{T}
\frac{1}{f_p(u)}\prob{\sup_{t\in[0,f_p(u)]} Q_{B_H}(t) > f_p(u)}\D u
\right),
\end{align*}
provided that $S$ is sufficiently large.
\vb

\noindent\textit{Estimate of $P_2'$.}

\vb

Clearly, for $j\ge i+2$ and any $0\le l\le L_i$, $0\le p\le L_j$; c.f. \eqref{eq:def_ai2},
\[
y_j s_{j,p} - y_i s_{i,l} = a_j + y_j p q_j - \left(a_i + y_i l  q_i\right) \ge (j-i - 1) y_i,
\]
so that by \eqref{eq:defr*}, for any $0\le i<j\le I$,
\begin{equation}
\label{eq:est1}
r^*_{i,j}:=\sup_{
	\substack{		
		0\le l\le L_i,
		0\le p\le L_j\\
		|n|\le N_i,
		|m|\le N_j
	}
} |\tilde r_{y_i,y_j}(s_{i,l},\tau_{i,n},s_{j,p},\tau_{j,m})|
\le
 r^*(j-i-1)\le r^*(1)<1.
\end{equation}
On the other hand, by
\eqref{eq:asymptotic}, there exist positive constants $s_0$, such that for sufficiently large $S$,
\begin{align}
\label{eq:est2}
(\tilde r_{y_i,y_j}(s_{i,l},\tau_{i,n},s_{j,p},\tau_{j,m}))^+ = 0, &\quad \text{if}\quad
j=i+1, \quad |y_j s_{j,p} - y_i s_{i,l}|/y_i  \le s_0,\\
\label{eq:est3}
|\tilde r_{y_i,y_j}(s_{i,l},\tau_{i,n},s_{j,p},\tau_{j,m})| \le
r^*(s_0)<1,
&\quad\text{if}\quad
j=i+1,\quad |y_j s_{j,p} - y_i s_{i,l}|/y_j > s_0
\end{align}
Therefore, by \eqref{eq:est1}--\eqref{eq:est3} we obtain
\begin{align*}
P_2' &\le
\sum_{\substack{0\le i\le I-1 \\ j = i+1}}\sum_{\substack{0\le l\le L_i\\0\le p\le L_j}}\sum_{\substack{|n|\le N_i\\|m|\le N_j}}
\frac{1}{\sqrt{1- r^*(s_0)}}
\exp\left(-\frac{\half(\hat v_i^2+\hat v_j^2)}{1+r^*(s_0)}\right)\\
&\quad +
\sum_{\substack{0\le i\le I-2 \\ i+2 \le j \le I }}\sum_{\substack{0\le l\le L_i\\0\le p\le L_j}}\sum_{\substack{|n|\le N_i\\|m|\le N_j}}
\frac{r^*(j-i-1)}{\sqrt{1- r^*(1)}}
\exp\left(-\frac{\half(\hat v_i^2+\hat v_j^2)}{1+r^*(j-i-1)}\right).
\end{align*}
Completely \k2{similarly} to the estimation of $P_2$ in the proof of \autoref{lem:bound}, we can
\k2{get} that there exist positive constants $K$ and $\rho$ such that, for sufficiently large $S$,
\[
P_2'\le K S^{-\rho}.
\]
\end{proof}

The next lemma is a straightforward modification of \citep[Lemma 3.1 and Lemma 4.1]{Watanabe70}, see also \citep[Lemma 1.4]{Qualls71}.
\begin{lem}
\label{lem:v2}
If \autoref{thm:equiv} is true under the additional condition, that for large $t$,
\begin{equation}
\label{eq:frestricted}
\left(\frac{2}{A^2} \log t \right)^{1/2(1-H)} \le f(t) \le  \left(\frac{3}{A^2} \log t \right)^{1/2(1-H)},
\end{equation}
it is true without the additional condition.
\end{lem}

\section{Proof of the main results}
\label{sec:Proofs}
\begin{proof}[\bf{Proof of \autoref{thm:equiv}}]
\noindent Note that the case $\mathscr I_f <\infty$ is straightforward and does not need any
additional knowledge on the process $Q_{B_H}$ apart from the stationarity property.
Indeed, consider the sequence of intervals
$M_i$ as in \autoref{lem:lbound}. Then, for any $\varepsilon>0$ and sufficiently large $T$,
\[
\sum_{k = [T]+1}^\infty
\prob{\sup_{t\in M_{k}} Q_{B_H}(t) > f(a_{k})}
=
\sum_{k= [T]}^\infty \prob{\sup_{t\in[0, f(b_k)]} Q_{B_H}(t) > f(b_k)}
\le
\mathscr I_f <\infty,
\]
and the Borel-Cantelli lemma completes this part of the proof since $f$ is an increasing function.

Now let $f$ be an increasing \K{function} such that $\mathscr I_f\equiv\infty$.
\K{Using} the same notation as in
\autoref{lem:bound} with $f$ instead of $f_p$, we find that, for any $S,\varepsilon,\theta>0$,
\begin{align*}
\prob{Q_{B_H}(s) > f(s)\text{ i.o.}} &\ge \prob{\left\{\sup_{t\in I_k} Q_{B_H}(t) > f(t_k)\right\}\quad\text{i.o.}}\\
&\ge
\prob{ \left\{ \max_{\substack{ 0\le l\le L_k\\\quad 0\le |n| \le N_k}} AZ_{x_k}(s_{k,l},\tau_{k,n})> v_k\right\}\text{ i.o.}}.
\end{align*}
Let
\[
E_k = \left\{ \max_{\substack{ 0\le l\le L_k\\\quad 0\le |n| \le N_k}} AZ_{x_k}(s_{k,l},\tau_{k,n})\le v_k\right\}.
\]
For sufficiently large $S$ and $\theta$; c.f. estimation of $P_1$, we get
\begin{equation}
\label{eq:sum}
\sum_{k=0}^\infty
\prob{ E_k^c}
\ge
\frac{(1-\varepsilon)}{(1+\varepsilon)}
\int_{S+f(S)}^{\infty}
\frac{1}{f(u)}\prob{\sup_{t\in [0,f(u)]} Q_{B_H}(t) >  f(u)}\D u
=\infty.
\end{equation}
Note that
\[
1-\prob{E_i^c\quad \text{i.o.}}
=
\lim_{m\toi}\prod_{k=m}^\infty\prob{E_k} +
\lim_{m\toi}\left(\prob{\bigcap_{k=m}^\infty E_k} - \prod_{k=m}^\infty\prob{E_k}\right).
\]
The first limit \K{equals to} zero as a consequence of \eqref{eq:sum}.
\K{The second limit equals to} zero because of the asymptotic independence of the events $E_k$.
Indeed, there exist positive constants $K$ and $\rho$, depending only on $H,\varepsilon,\lambda$, such that for any $n>m$,
\[
A_{m,n}=\left|\prob{\bigcap_{k=m}^n E_k} - \prod_{k=m}^n\prob{E_k}\right|\le K (S+m)^{-\rho},
\]
by the same calculations as in the estimate of $P_2$ in \autoref{lem:bound} after realizing that, by \autoref{lem:v2},
we might restrict ourselves to the case when
\eqref{eq:frestricted} holds. Therefore $\prob{E_i^c \text{ i.o.}} =1$,
which \K{completes} the proof.
\end{proof}

\noindent \textbf{Proof of \autoref{thm:main}}
\K{In order to make the proof more transparent we divide it on several steps.}
\vb

\noindent\textit{Step 1.}
Let $p>1$. \K{Then,} for every $\varepsilon\in(0,\frac{1}{4})$,
\[
\liminf_{t\toi}\frac{\xi_p(t) -t}{h_p(t)}\ge -(1+2\varepsilon)^2\quad\text{a.s.}
\]
\begin{proof}
Let $\{T_k:k\ge 1\}$ be a sequence such that $T_k\toi$, as $k\toi$.
Put $S_k =T_k -(1+2\varepsilon)^2 h_p(T_k)$. Since $h_p(t)=O(t\log^{1-p} t \log_2 t)$, then, for $p>1$,
$S_k\sim T_k$, as $k\toi$, and from \autoref{lem:bound} it follows that
\begin{align*}
\mathbb P\Big (&\frac{\xi_p(T_k) - T_k}{h_p(T_k)}\le -(1+2\varepsilon)^2\Big )
=
\prob{\xi_p(T_k)\le S_k}
=
\prob{\sup_{S_k<t\le T_k}\frac{Q_{B_H}(t)}{f_p(t)}< 1}\\
&\le
\exp\left(
-\frac{(1-\varepsilon)}{(1+\varepsilon)}
\int_{S_k+f_p(S_k)}^{T_k}
\frac{1}{f_p(u)}
\prob{\sup_{t\in [0,f_p(u)]} Q_{B_H}(t) >  f_p(u)}\D u
\right)
+ 2K T_k^{-\rho}.
\end{align*}
Moreover, as $k\toi$,
\begin{align}
\label{eq:logasympt}
\nonumber
\int_{S_k+f_p(S_k)}^{T_k}&
\frac{1}{f_p(u)}
\prob{\sup_{t\in [0,f_p(u)]} Q_{B_H}(t) >  f_p(u)}\D u\\
&\sim
(1+2\varepsilon)^2 h_p(T_k)
\frac{1}{f_p(T_k)}
\prob{
	\sup_{t\in [0,f_p(T_k)]}
		Q_{B_H}(t) >  f_p(T_k)}
=
(1+2\varepsilon)^2 p \log_2 T_k.
\end{align}
Now take $T_k = \exp(k^{1/p})$. Then,
\[
\sum_{k=0}^\infty
\prob{\xi_p(T_k)\le S_k}
\le
2K\sum_{k=0}^\infty k^{-(1+\varepsilon/2)}<\infty.
\]
Hence, by the Borel-Cantelli lemma, \K{we have}
\begin{equation}
\label{eq:step}
\liminf_{k\toi}\frac{\xi_p(T_k)-T_k}{h_p(T_k)}\ge -(1+2\varepsilon)^2\quad\text{a.s.}.
\end{equation}
Since $\xi_p(t)$ is a non-decreasing random function of $t$, for every $T_k\le t\le T_{k+1}$, we have
\begin{align*}
\frac{\xi_p(t)-t}{h_p(t)}\ge
\frac{\xi_p(T_k)-T_k}{h_p(T_k)}- \frac{T_{k+1}-T_k}{h_p(T_k)}.
\end{align*}
For $p >1$ elementary calculus implies
\[
\lim_{k\toi}\frac{T_{k+1}- T_k}{ h_p (T_k)}  =0,
\]
so that
\[
\liminf_{t\toi}\frac{\xi_p(t)-t}{h_p(t)}\ge\liminf_{k\toi}\frac{\xi_p(T_k)-T_k}{h_p(T_k)}\quad\text{a.s.},
\]
which \K{completes} the proof of this step.
\end{proof}

\noindent\textit{Step 2.}
Let $p>1$. Then, for every $\varepsilon\in(0,1)$,
\[
\liminf_{t\toi}\frac{\xi_p(t)-t}{h_p(t)}\le -(1-\varepsilon)\quad\text{a.s.}
\]
\begin{proof}
As in the proof of the lower bound \K{(Step 1), we} put
\[
T_k =\exp(k^{(1+\varepsilon^2)/p}),\quad S_k =T_k -(1-\varepsilon) h_p(T_k),\quad k\ge 1.
\]
Let
\[
B_k = \{\xi_p(T_k)\le S_k\}=\left\{\sup_{S_k<t\le T_k}\frac{Q_{B_H}(t)}{f_p(t)}<1\right\}.
\]
It suffices to show $\prob{B_n \text{ i.o.}}=1$, that is
\begin{equation}
\label{eq:ub_goal}
\lim_{m\toi}\prob{\bigcup_{k=m}^\infty B_k} =1.
\end{equation}
Let
\[
a_0^k = S_k, \quad y_0^k = f_p(a_0^k),\quad b_0^k = a_0^k + y_0^k,
\]
\[
a_i^k = b_{i-1}^k,\quad y_i^k = f_p(a_i^k),\quad  b_i^k = a_i^k +y_i^k,\quad M_i^k = (a_i^k, b_i^k],\quad
v_i^k = A (y_i^k)^{1-H},\quad
\tilde M_i^k = \frac{M_i^k}{ y_i^k}=(\tilde a_i^k, \tilde b_i^k].
\]
Define $J_k$ to be the biggest number such that $b_{J_k-1}^k\le T_k$ and $b_{J_k}^k> T_k$. In what follows let $b_{J_k}^k$ be redefined to $T_k$. Note that $J_k\le [(T_k-S_k)/f_p(S_k)]$.

Since $f_p$ is an increasing function,
\[
B_k = \bigcap_{i=0}^{J_k} \left\{\sup_{t\in M_i^k}\frac{Q_{B_H}(t)}{f_p(t)}<1\right\}
\supset \bigcap_{i=0}^{J_k} \left\{\sup_{t\in M_i^k}Q_{B_H}(t)< y_i^k\right\}
= \bigcap_{i=0}^{J_k} \left\{\sup_{\substack{s\in \tilde M_i^k \\ \tau\ge 0}} AZ_{y_i^k}(s,\tau)<
v_i^k\right\}.
\]
Analogously to \eqref{eq:disc2}, define a discretization of the set $\tilde M_i^k\times J(v_i^k)$ as follows
\begin{align*}
s_{i,l}^k &= \tilde a_i^k+ l q_i^k,\quad
0\le l\le L^k_i,\quad  L^k_i = [1/ q_i^k],
\quad q_i^k  = \theta_i^k (v_i^k)^{-\frac{1}{H}},
\quad \theta_i^k  = (v_i^k)^{-\frac{4}{H}}\\
\tau_{i,n}^k &= \tau_0 + n q_i^k,\quad 0\le |n| \le N^k_i,\quad N^k_i = [\tau^*(v_i^k)/q_i^k].
\end{align*}
Finally, let
\[
A_k=\bigcap_{i=0}^{J_k}\left\{\max_{
\substack{0\le l\le L^k_i\\0\le |n| \le N^k_i}} AZ_{y_i^k}(s_{i,l}^k,\tau_{i,n}^k)\le v_i^k-\frac{(\theta_i^k)^{\frac{H}{2}}}{v_i^k}\right\}.
\]
Observe that
\[
\prob{\bigcup_{k=m}^\infty A_k} \le \prob{\bigcup_{k=m}^\infty B_k}+\sum_{k=m}^\infty\prob{A_k\cap B_k^c}.
\]
Furthermore,
\begin{align}
\sum_{k=m}^\infty\prob{A_k\cap B_k^c} &\le \sum_{k=m}^\infty\sum_{i=0}^{J_k}
\prob{\max_{
\substack{0\le l\le L^k_i\\0\le |n| \le N^k_i}} AZ_{y_i^k}(s_{i,l}^k,\tau_{i,n}^k)\le v_i^k-\frac{(\theta_i^k)^{\frac{H}{2}}}{v_i^k}, \sup_{\substack{s\in \tilde M_i^k \\ \tau\ge 0}} AZ_{y_i^k}(s,\tau)\ge v_i^k}
\nonumber\\
&\le
\sum_{k=m}^\infty\sum_{i=0}^{J_k}
\prob{
	\max_{
\substack{0\le l\le L^k_i\\0\le |n| \le N^k_i}} AZ_{y_i^k}(s_{i,l}^k,\tau_{i,n}^k)\le v_i^k-\frac{(\theta_i^k)^{\frac{H}{2}}}{v_i^k},
    \sup_{
    	\substack{
    		s\in \tilde M_i^k \\
    		\tau\in J(v_i^k)}
    	} AZ_{y_i^k}(s,\tau)\ge v_i^k}
\nonumber\\
\label{eq:dummy}
&\quad +
\sum_{k=m}^\infty\sum_{i=0}^{J_k}
\prob{
	\sup_{
		\substack{
			s\in \tilde M_i^k \\
			\tau \notin J(v_i^k)
		}
	}
		AZ_{y_i^k}(s,\tau)\ge v_i^k
}.
\end{align}
By \autoref{lem:discreate_approx}, for sufficiently large $m$ and some $K_1, K_2>0$, the first sum is bounded from above by
\begin{align*}
K&\sum_{k=m}^\infty\sum_{i=0}^{J_k}
(v_i^k)^{\frac{2(1-2H)}{H}} \exp\left(-\frac{(v_i^k)^2}{2}-\frac{(v_i^k)^{4}}{K_1}\right)
\le
K\sum_{k=m}^\infty\sum_{i=0}^{J_k}
\frac{(\log a_i^k)^{\frac{1-2H}{H}}}{a_i^k(\log a_i^k)^{1+C_H-p}}
\exp\left(-\frac{\log^2( a_i^k)}{K_2}\right)\\
&\le
K\sum_{k=m}^\infty\sum_{i=0}^{J_k}
(S_k + i f_p(S_k))^{-2}
\le
K\sum_{k=m}^\infty
(S_k )^{-1}
\le K m^{-4}.
\end{align*}
Note that by \eqref{eq:tauBigger}, for sufficiently large $m$, the term in \eqref{eq:dummy} is bounded from above by
\begin{align*}
K&\sum_{k=m}^\infty\sum_{i=0}^{J_k} (v_i^k)^{\frac{2}{H}}
\exp\left(-\half (v_i^k)^2 - b \log^2 v_i^k\right)
\le
K\sum_{k=m}^\infty\sum_{i=0}^{J_k}
\frac{1}{a_i^k \log^{1+p} a_i^k}
\le
K\sum_{k=m}^\infty\sum_{i=S_k}^{\infty}
\frac{1}{i \log^{1+p} i}
\\
&\le
K
\sum_{k=m}^\infty (\log(S_k))^{-p}\le
K\sum_{k=m}^\infty k^{-(1+\varepsilon)^2}\le K m^{-\varepsilon}.
\end{align*}
Therefore
\[
\lim_{m\toi} \sum_{k=m}^\infty\prob{A_k\cap B_k^c} = 0
\]
and
\[
\lim_{m\toi}\prob{\bigcup_{k=m}^\infty B_k}\ge\lim_{m\toi} \prob{\bigcup_{k=m}^\infty A_k} .
\]
\K{In order to complete} the proof of \eqref{eq:ub_goal} we only need to show that
\begin{equation}
\label{eq:generalBC}
\prob{A_n \text{ i.o.}}=1.
\end{equation}
Similarly to \eqref{eq:logasympt}, we have
\[
\int_{S_k}^{T_k}
\frac{1}{f_p(u)}
\prob{\sup_{t\in [0,f_p(u)]} Q_{B_H}(t) >  f_p(u)}\D u
	\sim
(1-\varepsilon) p \log_2 T_k.
\]
Now from \autoref{lem:lbound} it follows that
\[
\prob{A_k} \ge
\frac{1}{4}\exp\left(
-(1-\varepsilon^2)
p \log_2 T_k
\right) - K S_k^{-\rho}
\ge
\frac{1}{8}k^{-(1-\varepsilon^4)},
\]
for every $k$ sufficiently large. Hence,
\begin{equation}
\label{eq:sumA}
\sum_{k=1}^\infty \prob{A_k} =\infty.
\end{equation}

Applying Berman's inequality, we get for $t<k$
\begin{equation}
\label{eq:AkAt}
\prob{A_kA_t}\le \prob{A_k}\prob{A_t} + Q_{k,t},
\end{equation}
where
\[
Q_{k,t} =  \sum_{\substack{0\le i\le J_k \\ 0\le j \le J_t}}\sum_{\substack{0\le l\le L_i^k\\0\le p\le L_j^t}}\sum_{\substack{|n|\le N_i^k\\|m|\le N_j^t}}
\frac{|r_{y_i^k,y_j^t}(s_{i,l}^k,\tau_{i,n}^k,s_{j,p}^t,\tau_{j,m}^t)|}{\sqrt{1- r^2_{y_i^k,y_j^t}(s_{i,l}^k,\tau_{i,n}^k,s_{j,p}^t,\tau_{j,m}^t)}}
\exp\left(
-\frac{( v_i^k-(v_i^k)^{-3})^2+( v_j^t-(v_j^t)^{-3})^2}{2(1+ |r_{y_i^k,y_j^t}(s_{i,l}^k,\tau_{i,n}^k,s_{j,p}^t,\tau_{j,m}^t)|)}
\right).
\]
For any $0\le i\le J_k$,
$0\le j\le J_t$, $0\le l\le L_i^k$, $0\le p\le L_j^t$, and $t<k$,
\[
y_i^k s_{i,l}^k-y_j^t s_{j,p}^t =
a_i^k + y_i^k l  q_i^k- \left(a_j^t+ y_j^t p q_j^t\right)
\ge
S_k - T_t \ge S_k - T_{k-1}
\ge
\half(T_k-T_{k-1}),
\]
where the last inequality holds for $k$ large enough, \K{since} 
\[
\frac{S_{k+1}-T_k}{T_{k+1}-T_k}\sim 1,\quad\as k.
\]
Thus, for sufficiently large $k$ and every $0\le t<k$, c.f. \eqref{eq:defr*},
\begin{align*}
\sup_{
	\substack{
		0\le i\le J_k\\	
		0\le j\le J_t\\	
		0\le l\le L_i^k,
		0\le p\le L_j^t\\
		|n|\le N_i^k,
		|m|\le N_j^t
	}
}
&
|r_{y_i^k,y_j^t}(s_{i,l}^k,\tau_n^k,s_{j,p}^t,\tau_{j,m}^t)|
\le
\sup_{
	\substack{
		0\le i\le J_k\\	
		0\le j\le J_t
	}
}r^*\left(\frac{T_k-T_{k-1}}{2\sqrt{y_i^ky_j^k}}\right)\\
&
\le
\mathcal K \left(\frac{T_k - T_{k-1}}{2 f_p(T_k)}\right)^{-\lambda}
\le
\mathcal K (T_k - T_{k-1})^{-\lambda/2}
\le\frac{\min(1,\lambda)}{32}.
\end{align*}
Therefore, for some generic constant $K$ not depending on $k$ and $t$ which may vary between lines, for every $t<k$ sufficiently large,
\begin{align*}
Q_{k,t} &\le K \sum_{\substack{0\le i\le J_k \\ 0\le j \le J_t}}
L_i^kL_j^tN_i^kN_j^t (T_k-T_{k-1})^{-\lambda/2}
\exp\left(-\frac{(v_i^k)^2+(v_j^t)^2}{2(1+\frac{\lambda}{16})}\right)\\
&\le
K (T_k-T_{k-1})^{-\lambda/2}
(L_{J_k}^k L_{J_t}^t)^2
\sum_{\substack{0\le i\le J_k \\ 0\le j \le J_t}}
\left(
	a_i^k
	\log^{1+C_H-p} a_i^k
\right)^{-\frac{1}{1+\frac{\lambda}{16}}}
\left(
	a_j^t
	\log^{1+C_H-p}a_j^t
\right)^{-\frac{1}{1+\frac{\lambda}{16}}}
\\
&\le
K (T_k-T_{k-1})^{-\lambda/2}\log^{10} T_k
\left(T_k\right)^{\frac{\frac{\lambda}{8}}{1+\frac{\lambda}{8}}}
\left(T_t\right)^{\frac{\frac{\lambda}{8}}{1+\frac{\lambda}{8}}}
\\
&\le
K T_k^{-\lambda/8} \le K \exp(-\lambda k^{(1+\varepsilon^2)/p}/8).
\end{align*}
Hence, we have
\begin{equation}
\label{eq:sumC}
\sum_{0\le t<k<\infty} Q_{k,t} <\infty.
\end{equation}
Now \eqref{eq:generalBC} follows from \eqref{eq:sumA}-\eqref{eq:sumC} and the general form of the Borel-Cantelli lemma.
\end{proof}

\vb

\noindent\textit{Step 3.}
If $p=1$, then for every $\varepsilon\in(0,\frac{1}{4})$
\begin{equation}
\label{eq:step3eq1}
\liminf_{t\toi}\frac{\log\left(\xi_p(t)/t\right)}{h_p(t)/t}\ge -(1+2\varepsilon)^2\quad\text{a.s.}
\end{equation}
and
\begin{equation}
\label{eq:step3eq2}
\liminf_{t\toi}\frac{\log\left(\xi_p(t)/t\right)}{h_p(t)/t}\le -(1-\varepsilon)\quad\text{a.s.}
\end{equation}
\begin{proof}
Put
\[
T_k = \exp(k),\quad S_k = T_k \exp\left(-(1+2\varepsilon)^2h_p(T_k)\right).
\]
Proceeding the same as in the proof of \eqref{eq:step}, one can obtain that
\[
\liminf_{k\toi}\frac{\log\left(\xi_p(T_k)/T_k\right)}{h_p(T_k)/T_k}\ge -(1+2\varepsilon)^2\quad\text{a.s.}
\]
On the other \K{hand} it is clear that
\[
\liminf_{t\toi}\frac{\log\left(\xi_p(t)/t\right)}{h_p(t)/t}
=
\liminf_{k\toi}\frac{\log\left(\xi_p(T_k)/T_k\right)}{h_p(T_k)/T_k}\quad\text{a.s.},
\]
since
\[
\liminf_{k\toi}\frac{\log\left(T_k/T_{k+1}\right)}{h_p(T_k)/T_k}=0.
\]
This proves \eqref{eq:step3eq1}.

Let
\[
T_k = \exp\left(k^{1+\varepsilon^2}\right),\quad S_k = T_k\exp\left(-(1-\varepsilon)h_p(T_k)\right).
\]
Noting that
\[
\frac{S_{k+1}-T_k}{S_{k+1}}\sim 1,\quad\as k,
\]
\K{following} along the same lines as in the proof of \eqref{eq:ub_goal}, we also have
\[
\liminf_{k\toi}\frac{\log\left(\xi_p(T_k)/T_k\right)}{h_p(T_k)/T_k}\le -(1-\varepsilon)\quad\text{a.s.},
\]
which proves \eqref{eq:step3eq2}.
\end{proof}

{\bf Acknowledgement}:
\k2{We are thankful to the editor and the referee for several suggestions which improved our manuscript.}
K. D\polhk{e}bicki was partially supported by
National Science Centre Grant No. 2015/17/B/ST1/01102 (2016-2019).
Research of K. Kosi\'nski was conducted under scientific
Grant No. 2014/12/S/ST1/00491 funded by National Science Centre.

\small\bibliography{biblioteczka}

\end{document}